\documentclass{amsart}

\usepackage{amsfonts}
\usepackage{amssymb}
\usepackage{enumerate}
\usepackage{amsmath}
\usepackage{amsthm}
\usepackage{graphicx}
\usepackage[english]{babel}

\numberwithin{equation}{section}

\newtheorem{theorem}{Theorem}[section]
\newtheorem{lemma}[theorem]{Lemma}
\newtheorem{question}[theorem]{Question}
\newtheorem{problem}[theorem]{Problem}
\newtheorem{corollary}[theorem]{Corollary}

\newtheorem*{Bruckner}{Theorem \ref{t:npns}}
\newtheorem*{isoshy}{Theorem \ref{t:isoshy}}
\newtheorem*{perfect}{Theorem \ref{t:perfect}}
\newtheorem*{cperfect}{Corollary \ref{c:perfect}}
\newtheorem*{infderiv}{Theorem \ref{t:infderiv}}

\theoremstyle{definition}
\newtheorem{remark}[theorem]{Remark}
\newtheorem{definition}[theorem]{Definition}

\DeclareMathOperator{\dist}{dist}
\DeclareMathOperator{\inter}{int}
\DeclareMathOperator{\cl}{cl}
\DeclareMathOperator{\diam}{diam}
\DeclareMathOperator{\supp}{supp}
\DeclareMathOperator{\Prob}{Pr}
\DeclareMathOperator{\Comp}{Comp}

\begin{document}

\title{Bruckner--Garg-type results with respect to Haar null sets in $C[0,1]$}
\author{Rich\'ard Balka}
\address{Department of Mathematics, University of Washington, Box 354350, Seattle, WA 98195-4350, USA
and Alfr\'ed R\'enyi Institute of Mathematics, Hungarian Academy of Sciences, PO Box 127, 1364 Budapest, Hungary}
\email{balka@math.washington.edu}

\thanks{The first author was supported by the
Hungarian Scientific Research Fund grants no.~72655 and 104178. The third author was supported by
the Hungarian Scientific Research Fund grants no.~72655, 83726 and 104178.}

\author{Udayan B. Darji}

\address{Department of Mathematics, University of Louisville, Louisville, KY 40292,
USA}

\email{ubdarj01@louisville.edu}

\author{M\'arton Elekes}

\address{Alfr\'ed R\'enyi Institute of Mathematics, Hungarian Academy of Sciences,
PO Box 127, 1364 Budapest, Hungary and E\"otv\"os Lor\'and
University, Institute of Mathematics, P\'azm\'any P\'eter s. 1/c,
1117 Budapest, Hungary}

\email{elekes.marton@renyi.mta.hu}



\thispagestyle{empty}

\begin{abstract}
A set $\mathcal{A}\subset C[0,1]$ is \emph{shy} or \emph{Haar null } (in the sense of Christensen)
if there exists a Borel set $\mathcal{B}\subset C[0,1]$ and a Borel probability measure $\mu$ on $C[0,1]$ such that $\mathcal{A}\subset \mathcal{B}$ and
$\mu\left(\mathcal{B}+f\right) = 0$ for all $f \in C[0,1]$. The complement of a shy set is called a \emph{prevalent} set. We say that a set is \emph{Haar ambivalent} if it is neither shy nor prevalent.

The main goal of the paper is to answer the following question: What can we say about the topological properties of the level sets of the prevalent/non-shy many $f\in C[0,1]$?

The classical Bruckner--Garg Theorem characterizes the level sets of the generic (in the sense of Baire category)
$f\in C[0,1]$ from the topological point of view. We prove that the functions $f\in C[0,1]$ for which the same characterization holds form a Haar ambivalent set.

In an earlier paper we proved that the functions
$f\in C[0,1]$ for which positively many level sets with respect to the Lebesgue measure $\lambda$
are singletons form a non-shy set in $C[0,1]$. The above result yields that this set is actually Haar ambivalent.
Now we prove that the functions $f\in C[0,1]$ for which positively many level sets
with respect to the occupation measure $\lambda\circ f^{-1}$ are not perfect
form a Haar ambivalent set in $C[0,1]$.

We show that for the prevalent $f\in C[0,1]$ for the generic
$y\in f([0,1])$ the level set $f^{-1}(y)$ is perfect.

Finally, we answer a question of Darji and White by showing that the set of functions $f \in C[0,1]$ for which there exists a perfect
set $P_f\subset [0,1]$ such that $f'(x) = \infty$ for all $x \in P_f$ is Haar ambivalent.
\end{abstract}

\maketitle

\section{Introduction}

Let $G$ be a \emph{Polish group}, i.e. a separable topological group which is endowed with a compatible complete metric.
If $G$ is locally compact then there exists a \emph{Haar measure} on $G$, that is,
a left translation invariant regular Borel measure which is finite on compact sets and positive on non-empty open sets.
The concept of Haar measure does not extend to groups that are not locally compact, but the notion of Haar measure zero does. The following definition is due to Christensen \cite{C}
and was rediscovered by Hunt, Sauer and York \cite{HSY}.

\begin{definition} \label{d:shy} For an abelian Polish group $G$ a set $A\subset G$ is \emph{shy} or \emph{Haar null}
if there exists a Borel set $B\subset G$ and a Borel probability measure $\mu$ on $G$ such that $A\subset B$ and $\mu\left(B+x\right) = 0$ for all $x\in G$.
The complement of a shy set is called a \emph{prevalent} set. We say that a set is \emph {Haar ambivalent} if it is neither shy nor prevalent.
\end{definition}

Christensen proved in \cite{C} that shy sets form a $\sigma$-ideal and in locally compact abelian Polish groups Haar measure zero sets and shy sets coincide.

\bigskip

Denote by $C[0,1]$  the Banach space of continuous functions $f\colon [0,1]\to \mathbb{R}$ endowed with the supremum metric.
By \emph{Cantor set} we mean a set homeomorphic to the classical `middle-third' Cantor set and \emph{generic} is understood in the sense of Baire category.
Let us recall the well-known Bruckner--Garg Theorem, see \cite{BG}.

\begin{theorem}[Bruckner--Garg] \label{t:BG} The generic $f\in C[0,1]$ has the property that there is a countable
dense set $D_f\subset (\min f,\max f)$ such that
\begin{enumerate}[(1)]
\item $f^{-1}(y)$ is a singleton if $y\in  \{\min f,\max f\}$,

\item $f^{-1}(y)$ is a Cantor set if $y\in (\min f,\max f) \setminus D_f$,

\item $f^{-1}(y)$ is the union of a Cantor set and an isolated point if $y\in D_f$.
\end{enumerate}
\end{theorem}

The above theorem completely describes the level set structure of the generic $f\in C[0,1]$ from the topological point of view.
If we now replace Baire category with the measure theoretic
notion of prevalence, we arrive at the the main question of the paper:

\begin{question} What can we say about the topological properties of the level sets of the prevalent/non-shy many $f\in C[0,1]$?
\end{question}

We say that $f\in C[0,1]$ has the \emph{Bruckner--Garg property} if its level sets are as described in Theorem \ref{t:BG}. The following
theorem shows that our situation is more complicated than the Baire category case.

\begin{Bruckner} The set
$$\{f\in C[0,1]: f \textrm{ has the Bruckner--Garg property}\}$$
is Haar ambivalent in $C[0,1]$.
\end{Bruckner}

We use $\exists^{\mu}$ to denote positively many
with respect to the measure $\mu$. The following theorem is \cite[Corollary~6.2]{BDE}.

\begin{theorem}[Balka--Darji--Elekes] \label{t:nonshy}
The set
$$\{f\in C[0,1]:  \exists^{\lambda} y\in \mathbb{R} \textrm{ such that } f^{-1}(y) \textrm{ is a singleton}\}$$
is Haar ambivalent in $C[0,1]$.
\end{theorem}

For $f\in C[0,1]$ let $\lambda_f=\lambda\circ f^{-1}$ be \emph{the occupation measure corresponding to $f$}.
In Theorem~\ref{t:nonshy} one cannot replace Lebesgue measure with occupation measure, because we showed in \cite{BDE} that for the
prevalent $f\in C[0,1]$ for $\lambda_f$ almost every $y\in \mathbb{R}$ the level set $f^{-1}(y)$ has Hausdorff dimension $1$. The following theorem
yields that some of these level sets are not perfect.

\begin{isoshy} The set
$$\left\{f\in C[0,1]: \exists^{\lambda_f} y\in \mathbb{R} \textrm{ such that } f^{-1}(y) \textrm{ is not perfect}\right\}$$
is Haar ambivalent in $C[0,1]$.
\end{isoshy}

Next we consider maps from an uncountable compact metric space $K$ to $\mathbb{R}^d$. Let us denote by $C(K,\mathbb{R}^d)$ the set of continuous functions from $K$ to $\mathbb{R}^d$ endowed with the supremum metric. Prevalent continuous maps have many fibers of cardinality continuum, the following theorem is essentially
\cite[Theorem~11]{D} and the remark following its proof.

\begin{theorem}[Dougherty] \label{t:D} Let $K$ be the middle-third Cantor set and let $d\in \mathbb{N}^+$.
Then for the prevalent $f\in C(K,\mathbb{R}^d)$ there exists a non-empty open set $U_f\subset \mathbb{R}^d$ such that for all $y\in U_f$
$$\# f^{-1}(y)=2^{\aleph_0}.$$ \end{theorem}

Applying the above theorem we will show that generic fibers are perfect.

\begin{perfect} Let $K$ be a compact metric space without isolated points and let $d\in \mathbb{N}^+$. Then for the prevalent $f\in C(K,\mathbb{R}^d)$ for the generic $y\in f(K)$
$$f^{-1}(y) \textrm{ is perfect}.$$
\end{perfect}

\begin{cperfect} For the prevalent $f\in C[0,1]$ for the generic
$y\in f([0,1])$
$$f^{-1}(y) \textrm{ is perfect}.$$
\end{cperfect}

In \cite{DW} the following theorem was proved.

\begin{theorem}[Darji--White] Let $P \subseteq [0,1]$ be perfect. Then the set
\[\mathcal{D} _P = \{f \in C[0,1]: \forall x \in P, \ f'(x) =\infty \textrm{ for all } x\in P  \}
\]
is shy.
\end{theorem}

Darji and White asked in \cite{DW} whether the above theorem holds
if $P$ is allowed to vary with $f$. We answer their question in the negative.

\begin{infderiv}The set
\[\mathcal{D} = \{f \in C[0,1]: \exists \text{ a perfect set } P_f  \text { such that } f'(x) = \infty \textrm{ for all } x\in P_f \}
\] is Haar ambivalent.
\end{infderiv}

It is also worth mentioning that Simon \cite{S} also considered a measure-theoretic counterpart of the Bruckner--Garg Theorem, and he proved that, almost surely, the linear Brownian motion defined on $[0,1]$ has the same level set structure as the generic $f\in C[0,1]$.

\section{Preliminaries}

Let $(X,d)$ be a metric space. For $x\in X$ and $r>0$ let $B(x,r)$ and $U(x,r)$ be the closed and open balls of
radius $r$ centered at $x$, respectively. For $A\subset X$ we denote by $\cl A$,
$\inter A$ and $\partial A$ the closure, interior and boundary of $A$, respectively.
The diameter of $A$ is denoted by $\diam A$.
For $A,B \subseteq X$ let us define $\dist(A,B) = \inf\{d(x,y) : x\in A,~y\in B\}$.

Let $X$ be a \emph{complete} metric space. A set is \emph{somewhere dense} if
it is dense in a non-empty open set, and otherwise it is called \emph{nowhere dense}. We say that $M \subset X$ is
\emph{meager} if it is a countable union of nowhere dense sets, and
a set is called \emph{co-meager} if its complement is meager. We say that the \emph{generic} element $x \in X$ has
property $\mathcal{P}$ if $\{x \in X : x \textrm{ has property }
\mathcal{P} \}$ is co-meager.
A metric space X is \emph{Polish} if it is complete and separable. See e.g. \cite{K} for
more on these concepts. A set is \emph{perfect} if it is closed and has no isolated points.

For a measure $\mu$ we use $\exists^{\mu}$ to denote positively many
with respect to $\mu$. Let $\lambda$ be the one-dimensional Lebesgue measure and for all $f\in C[0,1]$ let
$\lambda_f=\lambda \circ f^{-1}$ be the \emph{occupation measure corresponding to $f$}.

For the following lemma see \cite[Proposition~8.]{D}.

\begin{lemma}\label{l:D} Let $G,H$ be abelian Polish groups and let $\Phi \colon G\to H$ be a continuous epimorphism.
If $S\subset H$ is prevalent then so is $\Phi^{-1}(S)\subset G$.
\end{lemma}

The next corollary follows from Lemma~\ref{l:D} and the fact that the Tietze Extension Theorem holds in $\mathbb{R}^d$.
\begin{corollary} \label{c:hereditary} Let $K_1\subset K_2$ be compact metric spaces, let $d\in \mathbb{N}^+$ and
define
$$R\colon C(K_2,\mathbb{R}^d)\to C(K_1,\mathbb{R}^d), \quad R(f)=f|_{K_1}.$$
If $\mathcal{A} \subset C(K_1,\mathbb{R}^d)$ is prevalent then so is $R^{-1}(\mathcal{A})\subset C(K_2,\mathbb{R}^d)$.
\end{corollary}

The next theorem follows from Theorem~\ref{t:D}, Corollary~\ref{c:hereditary} and the fact  that every uncountable Polish space contains a
compact set homeomorphic to the middle-third Cantor set, see \cite[Cor.~6.5]{K}.

\begin{theorem} \label{t:D2}
Let $K$ be an uncountable compact metric space and let $d\in \mathbb{N}^+$.
Then for the prevalent $f\in C(K,\mathbb{R}^d)$ there exists a non-empty open set $U_f\subset \mathbb{R}^d$ such that for all $y\in U_f$
$$\# f^{-1}(y)=2^{\aleph_0}.$$
\end{theorem}

\section{Bruckner--Garg-type theorem for prevalent continuous functions}

The goal of this section is to prove our main theorem, Theorem \ref{t:npns}.

\begin{definition} We say that $f\in C[0,1]$ has the \emph{Bruckner--Garg property} if its level sets are as described in Theorem \ref{t:BG}.
\end{definition}

\begin{definition}
We say that $f\in C[0,1]$ is \emph{non-decreasing at a point $x\in [0,1]$} if there exists $\varepsilon>0$ such that  $\frac{f(z) - f(x)}{z-x} \ge 0 $ for all $z \in [0,1]$ with $0 < |z-x| < \varepsilon$.  A function $f\in C[0,1]$ is \emph{non-increasing at $x$} if $-f$ is non-decreasing at $x$, and $f$ is \emph{monotone at $x$} if $f$ is either non-decreasing or non-increasing at $x$. We say that $x\in [0,1]$ is a \emph{knot point} of $f\in C[0,1]$ if the upper and lower derivatives at $x$ satisfy $\overline{D}(f)(x)=\infty$ and $\underline{D}(f)(x)=-\infty$, respectively. (Here $\overline{D}(f)(x)= \limsup_{z \to x} \frac{f(z) - f(x)}{z-x}$ and $\underline{D}(f)(x)= \liminf_{z \to x} \frac{f(z) - f(x)}{z-x}$.)
\end{definition}

\begin{theorem} \label{t:npns} The set
$$\mathcal{A}=\{f\in C[0,1]: f \textrm{ has the Bruckner--Garg property}\}$$
is Haar ambivalent in $C[0,1]$.
\end{theorem}

Before proving Theorem \ref{t:npns} we need three lemmas.
For the following one see the proof of \cite[Theorem~3.3]{BG}.

\begin{lemma} \label{l:BGp}
If $f\in C[0,1]$ is not monotone at any point and one-to-one on its local extremum points then $f$ has the Bruckner--Garg property.
\end{lemma}

The next theorem was proved by Darji and White, see \cite[Theorem~1.1]{DW}.

\begin{theorem}[Darji--White] \label{t:DW} Let $A,B$ be disjoint countable subsets of $[0,1]$. Then
\begin{align*} \mathcal{M}_{A,B}=\{&f\in C[0,1]: f'(x)=\infty \textrm{ for all } x\in A,~f'(x)=-\infty \textrm{ for all } x\in B, \\
 &\textrm{ and } x \textrm{ is a knot point of } f \textrm{ for all } x \notin A\cup B\}
\end{align*}
is Haar ambivalent.
\end{theorem}

\begin{remark}
The above theorem is stated only for nonempty sets $A,B$ in \cite{DW}, but the proof works verbatim without this restriction.
\end{remark}

Applying Theorem~\ref{t:DW} for $A=B=\emptyset$ yields the following lemma.

\begin{lemma} \label{l:notmon} The functions $f\in C[0,1]$ that are not monotone at any point form a non-shy set in $C[0,1]$.
\end{lemma}

\begin{lemma} \label{l:notex} The prevalent $f\in C[0,1]$ is one-to-one on its local extremum points.
\end{lemma}

\begin{proof}
Let
$$\mathcal{A}=\{f\in C[0,1]: f \textrm{ is \emph {not} one-to-one on its local extremum points}\}.$$
Let $\mathcal{I}$ be the family of closed rational subintervals of $[0,1]$. For all disjoint $I,J\in \mathcal{I}$ let
$\mathcal{A}_{I,J}$ be the set of functions $f\in C[0,1]$ for which the maximum or
the minimum of $f$ on $I$ is equal to either the maximum or the minimum of $f$
on $J$. Clearly
$$\mathcal{A}=\bigcup_{I,J\in \mathcal{I},\,I\cap J=\emptyset} \mathcal{A}_{I,J}.$$
Since the countable union of shy sets is shy, it is enough to prove that the sets $\mathcal{A}_{I,J}$ are all shy.

Let us fix disjoint $I,J\in \mathcal{I}$, and we need to prove that $\mathcal{A}_{I,J}$ is shy. As $\mathcal{A}_{I,J}$ is clearly closed, it is Borel.
We may assume that $I=[a,b]$ and $J=[c,d]$ such that $b<c$.
For all $u \in [0,1]$ let $g_u\in C[0,1]$ be defined as
$$g_u(x) = \begin{cases} 0 & \textrm{ if } x \in [0,b], \\
u & \textrm{ if } x\in [c,1], \\
\textrm{affine} & \textrm{ if } x\in [b,c].
\end{cases}
$$
Let us define the continuous map
$$\phi\colon [0,1]\to C[0,1], \quad \phi(u)=g_u,$$
and consider the Borel probability measure $\mu=\lambda\circ \phi^{-1}$ on $C[0,1]$.
Note that the support of our measure satisfies $\supp (\mu) = \{g_u : u \in [0,1]\}$.

Now it is sufficient to show that $\mu(\mathcal{A}_{I,J} + f) = 0$ for every
$f\in C[0,1]$. But for a fixed $f\in C[0,1]$ it is easy to see that
$\left(\mathcal{A}_{I,J} + f \right) \cap \supp (\mu)$ is actually finite, since
there are at most four $u$ such that $g_u - f \in \mathcal{A}_{I,J}$. As finite sets are $\mu$-null, this completes the proof.
\end{proof}

\begin{proof}[Proof of Theorem \ref{t:npns}]
Theorem \ref{t:nonshy} implies that the functions $f\in C[0,1]$ having the
Bruckner--Garg property do not form a prevalent set in $C[0,1]$.

Now we prove that they form a non-shy set. Let
\begin{align*} \mathcal{A}&=\{f\in C[0,1]: f \textrm{ is not monotone at any point}\},\\
\mathcal{B}&=\{f\in C[0,1]: f \textrm{ is one-to-one on its local extremum points}\}.
\end{align*}
Lemma \ref{l:BGp} implies that it is enough to prove that $\mathcal{A}\cap \mathcal{B}$ is non-shy. The set $\mathcal{A}$ is non-shy by Lemma \ref{l:notmon} and $\mathcal{B}$ is prevalent
by Lemma \ref{l:notex}. Therefore $\mathcal{A}\cap \mathcal{B}$ is non-shy, and the proof is complete.
\end{proof}

\section{Level sets with respect to the occupation measure}

The main goal of this section is to prove Theorem \ref{t:isoshy}. First we need some preparation.

\begin{lemma} \label{l:ocB} The set
$$\mathcal{A}=\{f\in C[0,1]: \lambda_f \textrm{ is absolutely continuous with respect to } \lambda\}$$
is Borel.
\end{lemma}

\begin{proof} Let $\mathcal{S}$ be the family of all finite collections of pairwise disjoint open rational intervals of $\mathbb{R}$. Then $\mathcal{S}$ is countable.
For $n\in \mathbb{N}^+$ and $S\in \mathcal{S}$ let
\begin{align*} \mathcal{S}_n&=\left\{S\in \mathcal{S}: \lambda(\cup S)<1/n\right\},\\
\mathcal{A}_{n,S}&=\{f\in C[0,1]:  \lambda_f(\cup S)<1/n\}.
\end{align*}
As $ \lambda_f $ is absolutely continuous with respect to $\lambda$ iff the function $x \mapsto \lambda_f((-\infty,x))$ is absolutely continuous, we obtain that
$$\mathcal{A}=\bigcap_{n=1}^{\infty} \bigcup_{k=1}^{\infty} \bigcap_{S\in \mathcal{S}_k} \mathcal{A}_{n,S}.$$
Thus it is enough to prove that $\mathcal{A}_{n,S}$ is Borel for an arbitrarily fixed $n\in \mathbb{N}^+$ and $S=\{I_1,\dots,I_m\}\in \mathcal{S}$. For each open set
$U\subset \mathbb{R}$ consider $$\Phi_U \colon C[0,1]\to [0,1],\quad \Phi_U (f)=\lambda(f^{-1}(U)).$$
It is easy to see that
$$ \mathcal{A}_{n,S}=\left\{f\in C[0,1]: \sum_{i=1}^{m} \Phi_{I_i}(f)<\frac 1n\right\},$$
therefore it is enough to prove that $\sum_{i=1}^{m} \Phi_{I_i}$ is Borel measurable. Thus it suffices to show that $\Phi_U$ is Borel measurable
for every open set $U\subset \mathbb{R}$. Fix an arbitrary open set $U\subset \mathbb{R}$ and $r>0$, we will check that $\Phi_U^{-1}((r,\infty))$ is open. Pick $f\in \Phi_U^{-1}((r,\infty))$, we need to find $\varepsilon>0$ such that  $U(f,\varepsilon)\subset \Phi_U^{-1}((r,\infty))$. Since $\Phi_U(f)=\lambda(f^{-1}(U))>r$, the regularity of Lebesgue measure implies that there is a compact set $K\subset f^{-1}(U)$ such that $\lambda(K)>r$. As $f(K)\subset U$ is compact, we can define $\varepsilon=\dist(f(K),\mathbb{R}\setminus U)>0$. Clearly $g(K)\subset U$
for all $g\in U(f,\varepsilon)$, thus $\lambda(g^{-1}(U))\geq \lambda(K)>r$.
Hence $U(f,\varepsilon)\subset \Phi_U^{-1}((r,\infty))$, and the proof is complete.
\end{proof}

The following theorem is essentially known. However, for the sake of completeness
we point out how standard arguments concerning the Brownian motion yield this result.

\begin{theorem} \label{t:occ} For the prevalent $f\in C[0,1]$ the occupation measure $\lambda_f$ is absolutely continuous with respect to the Lebesgue measure $\lambda$.
\end{theorem}

\begin{proof} Consider
$$\mathcal{A}=\{f\in C[0,1]:  \lambda_f  \textrm{ is absolutely continuous with respect to } \lambda\},$$
then $\mathcal{A}$ is Borel by Lemma \ref{l:ocB}. Let $\mu$ be the Wiener measure on $C[0,1]$ and let $\{B(s): s\in [0,1]\}$ be the standard one-dimensional
Brownian motion. It is enough to prove
that $\mu(\mathcal{A}-f)=1$ for all $f\in C[0,1]$, that is, $\lambda_{B+f}$ is almost surely absolutely continuous with respect to  $\lambda$. Now one can repeat the proof of \cite[Theorem~3.26]{MP} with $t=1$, using also that for all $s_1,s_2,r\in [0,1]$
\begin{equation}\label{eq:Bs1} \Prob\left(|(B+f)(s_1)-(B+f)(s_2)|\leq r\right)\leq \Prob\left(|B(s_1)-B(s_2)|\leq r\right),
\end{equation}
which we verify next. Let $X$ be a standard normal random variable. We may assume that $s_1\neq s_2$ and let us
consider $a=\sqrt{|s_1-s_2|}$, $b=\frac{r}{a}$ and $c=\frac{f(s_1)-f(s_2)}{a}$. As $B(s_1)-B(s_2)$ has the same distribution as that of $aX$ and the
density function of $X$ is even and monotone decreasing on $[0,\infty)$, we obtain
\begin{align*} \Prob \left(|(B+f)(s_1)-(B+f)(s_2)|\leq r\right)&=\Prob(X\in [-b-c,b-c])\\
&\leq \Prob (X\in [-b,b]) \\
&=\Prob \left(|B(s_1)-B(s_2)|\leq r\right),
\end{align*}
thus \eqref{eq:Bs1} holds.
\end{proof}

\begin{theorem} \label{t:isoshy} The set
$$\mathcal{A}=\left\{f\in C[0,1]: \exists^{\lambda_f} y\in \mathbb{R} \textrm{ such that } f^{-1}(y) \textrm{ is not perfect}\right\}$$
is Haar ambivalent in $C[0,1]$.
\end{theorem}

\begin{proof} Theorem~\ref{t:npns} easily yields that $\mathcal{A}$ is not prevalent, so it is enough to prove that $\mathcal{A}$ is non-shy.
For all $f\in C[0,1]$ consider
$$S_f=\{y\in \mathbb{R}: f^{-1}(y)\cap [0,1/2) \textrm{ is a singleton}\}.$$
Then  $S_f$ is Borel, because it is easy to see that $\{y\in \mathbb{R}: \#(f^{-1}(y)\cap [0,1/2))\geq 1\}$ and
$\{y\in \mathbb{R}: \#(f^{-1}(y)\cap [0,1/2))\geq 2\}$ are $F_{\sigma}$ sets.
Theorem \ref{t:nonshy} and symmetry imply that
$$\mathcal{B}=\{f\in C[0,1]: \lambda(S_f)>0\}$$
is non-shy. Theorem \ref{t:occ} yields that
$$\mathcal{C}=\{f\in C[0,1]: \lambda_f \textrm{ is absolutely continuous with respect to } \lambda\}$$
is prevalent, so $\mathcal{B}\cap \mathcal{C}$ is non-shy. Assume that $f\in \mathcal{B}\cap \mathcal{C}$. Then $\lambda(S_f)>0$, so the Lebesgue Density Theorem \cite[223B Corollary]{Fr} implies that $S_f+\mathbb{Q}$ has full Lebesgue measure. Therefore the absolute continuity of $\lambda_f$ with
respect to $\lambda$ yields that $\lambda_f(S_f+\mathbb{Q})=1$. Thus there exists a $q(f)\in \mathbb{Q}$ such that
\begin{equation} \label{eq:Sf} \lambda\left(f^{-1}(S_f+q(f))\cap [2/3, 1]\right)>0.
\end{equation}
As shy sets form a $\sigma$-ideal, there is a $q\in \mathbb{Q}$ such that
$$\mathcal{D}=\{f\in \mathcal{B}\cap \mathcal{C}: q(f)=q\}$$
is non-shy. Set $I=[0,1/2]$ and $J=[2/3,1]$ and define $g\in C[0,1]$ as
$$g(x)=\begin{cases} 0 & \textrm{ if } x\in I, \\
q & \textrm{ if } x\in J, \\
\textrm{affine} & \textrm{ otherwise}.
\end{cases} $$
Since $\mathcal{D}$ is non-shy and shy sets are invariant under translations, $\mathcal{D}-g$ is also non-shy. Thus it is enough to prove that $\mathcal{D}-g\subset \mathcal{A}$.
Let us fix $f\in \mathcal{D}$, we prove that $f-g\in \mathcal{A}$. It is sufficient to show that
$\lambda_{f-g}(S_{f-g})>0$. Then $g|_{I}\equiv 0$, $g|_{J}=q$ and $q(f)=q$ imply that $S_{f-g}=S_f$ and
\begin{align*} (f-g)^{-1}(S_{f-g})\cap J&=(f-q)^{-1}(S_f)\cap J\\
&=f^{-1}(S_f+q(f))\cap J.
\end{align*}
The above equation and \eqref{eq:Sf} yield
\begin{align*}
\lambda_{f-g}(S_{f-g})&\geq \lambda((f-g)^{-1}(S_{f-g})\cap J)\\
&=\lambda\left(f^{-1}(S_f+q(f))\cap J\right)>0.
\end{align*}
This concludes the proof.
\end{proof}

\section{Generic level sets are perfect}

The aim of this section is to prove the following theorem.

\begin{theorem} \label{t:perfect}  Let $K$ be a compact metric space without isolated points and let $d\in \mathbb{N}^+$. Then for the prevalent $f\in C(K,\mathbb{R}^d)$ for the generic $y\in f(K)$
$$f^{-1}(y) \textrm{ is perfect}.$$
\end{theorem}

\begin{proof} First we prove that
\begin{align*} \mathcal{A}=\{&f\in C(K,\mathbb{R}^d): \exists \textrm{ an open set } U_f\subset \mathbb{R}^d \textrm{ such that } \\
&U_f \textrm{ is a dense subset of } f(K) \textrm{ and } \#f^{-1}(y)= 2^{\aleph_0} \textrm{ for all } y\in U_f\}
\end{align*}
is prevalent in $C(K,\mathbb{R}^d)$.
Let $\mathcal{V}=\{V_n: n\in \mathbb{N}^+\}$ be a countable basis of $K$ consisting of non-empty open sets.
For all $n\in \mathbb{N}^+$ let $K_n=\cl V_n$ and consider
\begin{align*} \mathcal{A}_n=\{&f\in C(K_n,\mathbb{R}^d): \exists \textrm{ a non-empty open set } U_f\subset \mathbb{R}^d \\
&\textrm{such that } \#f^{-1}(y)= 2^{\aleph_0} \textrm{ for all } y\in U_f \}.
\end{align*}
Since $K$ has no isolated points, the same holds for all $K_n$, hence they are uncountable by \cite[Cor.~6.3]{K}.
Thus Theorem~\ref{t:D2} implies that the $\mathcal{A}_n$ are prevalent. For all $n\in \mathbb{N}^+$ let us define
$$R_n\colon C(K,\mathbb{R}^d)\to C(K_n,\mathbb{R}^d), \quad R(f)=f|_{K_n}.$$
Corollary \ref{c:hereditary} implies that $R^{-1}_n(\mathcal{A}_n)$ are prevalent in $C(K,\mathbb{R}^d)$. As a countable intersection of prevalent sets,
$\bigcap_{n=1}^{\infty} R^{-1}_n(\mathcal{A}_n)$ is also prevalent in $C(K,\mathbb{R}^d)$. Thus it is enough to prove that $\bigcap_{n=1}^{\infty} R^{-1}_n(\mathcal{A}_n)\subset \mathcal{A}$.
Fix $f\in \bigcap_{n=1}^{\infty} R^{-1}_n(\mathcal{A}_n)$ and define
$$U_f=\bigcup_{n=1}^{\infty} U_{f|_{K_n}}.$$
Then $U_f\subset \mathbb{R}^d$ is open and is a  dense subset of $f(K)$. Clearly for all $y\in U_f$ there is an $n\in \mathbb{N}^+$ such that $y\in U_{f|_{K_n}}$, so
$$\#f^{-1}(y)\geq \#(f|_{K_n})^{-1}(y)=2^{\aleph_0}.$$
Hence $\mathcal{A}$ is prevalent in $C(K,\mathbb{R}^d)$.

Now for all $n\in \mathbb{N}^+$ let
\begin{align*} \mathcal{B}_n=\{&f\in C(K_n,\mathbb{R}^d): \exists \textrm{ an open set } W_{f}\subset \mathbb{R}^d \textrm{ such that } \\
&W_f \textrm{ is a dense subset of } f(K_n) \textrm{ and } \#f^{-1}(y)=2^{\aleph_0} \textrm{ for all } y\in W_{f}\}.
\end{align*}
Since the $K_n$ have no isolated points, the $\mathcal{B}_n$ are prevalent as above. Corollary \ref{c:hereditary} implies that the $R^{-1}_n(\mathcal{B}_n)$ are prevalent in $C(K,\mathbb{R}^d)$. As a countable intersection of prevalent sets, $\mathcal{B}=\bigcap_{n=1}^{\infty} R^{-1}_n(\mathcal{B}_n)$ is also prevalent in $C(K,\mathbb{R}^d)$. For all $f\in \mathcal{B}$ and $n\in \mathbb{N}^+$ let
$$W_{f,n}=W_{f|_{K_n}}\cup (f(K)\setminus f(K_n))$$
and
$$W_f=\bigcap_{n=1}^{\infty} W_{f,n}.$$
As a countable intersection of dense relatively open sets, $W_f$ is co-meager in $f(K)$. Let us fix $f\in \mathcal{B}$ and $y\in W_f$, it is enough to prove that $f^{-1}(y)$ is perfect. By definition, $y\in W_{f,n}$ for all $n\in \mathbb{N}^+$. If $y\in f(K)\setminus f(K_n)$ then $f^{-1}(y)\cap K_n=\emptyset$,
if $y\in W_{f|_{K_n}}$ then $\#(f^{-1}(y)\cap K_n)=2^{\aleph_0}$.
Thus $\#(f^{-1}(y)\cap K_n)\neq 1$ for all $n\in \mathbb{N}^+$, hence $f^{-1}(y)$ has no isolated point.
Therefore $f^{-1}(y)$ is perfect and the proof is complete.
\end{proof}

\begin{corollary} \label{c:perfect} For the prevalent $f\in C[0,1]$ for the generic
$y\in f([0,1])$
$$f^{-1}(y) \textrm{ is perfect}.$$
\end{corollary}

\section{Infinite derivative on perfect sets}

The main goal of this section is to prove the following theorem.
\begin{theorem}\label{t:infderiv}
The set
$$\mathcal{D} = \{f \in C[0,1]: \exists \text{ a perfect set } P_f  \text { such that } f'(x) = \infty \textrm{ for all } x\in P_f \}$$
is Haar ambivalent.
\end{theorem}

We need some preparation before we prove the theorem. The next lemma is well-known, see e.g. \cite[Lemma~4]{Z} for the proof.

\begin{lemma} \label{l:nshy} Let $G$ be an abelian Polish group and let $A\subset G$. If for all compact set $K\subset G$ there exists a $g\in G$ such that
$K+g\subset A$ then $A$ is non-shy.
 \end{lemma}

The following lemma is probably known, but we could not find a reference, so
we outline its short proof.

\begin{lemma} \label{l:iK} Let $\mathcal{K}\subset C[0,1]$ be a compact set. Then there is a strictly increasing subadditive function
$h\in C[0,1]$ such that $h(0)=0$ and for all $f\in \mathcal{K}$ and $x,z\in [0,1]$, $x\neq z$
$$|f(x)-f(z)|<h(|x-z|).$$
\end{lemma}

\begin{proof} Let us define $g\in C[0,1]$ as
\begin{align*} g(t)&=\sup_{f\in \mathcal{K}} M(f,t), \textrm{ where} \\
M(f,t)&=\sup \{|f(x)-f(z)|: x,z\in [0,1],\, |x-z|\leq t\}.
\end{align*}
By the Arzel\'a--Ascoli Theorem $\mathcal{K}$ is bounded and equicontinuous. Boundedness implies that $g(t)<\infty$ for all $t\in [0,1]$.
Clearly, $g$ is non-decreasing and for all $f\in C[0,1]$ and $s,t\in [0,1]$ we have
$M(f,t+s)\leq M(f,t)+M(f,s)$, so $g(t+s)\leq g(t)+g(s)$. Hence $g$ is subadditive.
By equicontinuity we obtain $\lim_{t\to 0+} g(t)=0$, so the subadditivity of $g$
yields
$$\lim_{s\to t} |g(t)-g(s)|\leq \lim_{s\to t} g(|t-s|)=0,$$
thus $g$ is continuous. Finally, let us define $h\in C[0,1]$ as
$$h(t)=g(t)+t.$$
The definition and properties of $g$ imply that $h$ satisfies the required conditions.
\end{proof}

\begin{lemma}\label{l:ih}
Let $h\in C[0,1]$ be a strictly increasing subadditive function with $h(0)=0$ and extend $h$ to $[-1,1]$ by $h(-x) = -h(x)$ for all $x \in [0,1]$.
Then there is a non-empty perfect $P\subset [0,1]$ and a non-decreasing $g \in C[0,1]$ such that for all $p\in P$
$$\lim_{x \rightarrow p} \frac{g(x) -g(p)}{h(x-p)}=\infty.$$
\end{lemma}

\begin{proof} We may assume that $h(x)\geq x$ for all $x\in [0,1]$, otherwise we may
add the identity function to it. For all $n\in \mathbb{N}^+$ let
\begin{equation} \label{eq:l_n} l_n=h^{-1}(5^{-n}),
\end{equation}
then $l_1\leq 1/5$. Since $h$ is subadditive on $[0,1]$,
we obtain that $l_{n+1}\leq l_n/5$ for all $n\in \mathbb{N}^+$.
Therefore we can define for all $n\in \mathbb{N}^+$ and $\sigma \in \{0,1,2,3\}^n$ closed intervals
$I_{\sigma}\subset [0,1]$ such that
\begin{enumerate}
\item \label{i:1} $\lambda(I_{\sigma})=l_n$,
\item \label{i:2} $I_{\sigma i} \subset I_{\sigma }$ for all $i\in \{0,1,2,3\}$,
\item \label{i:3} The intervals $\{I_{\sigma}\}_{\sigma\in \{0,1,2,3\}^n}$ are pairwise disjoint and they are placed
according to the lexicographical ordering of the indexes $\sigma$.
\end{enumerate}
Let us define $g\in C[0,1]$ as
\begin{align*} g(x)&=\int_{[0,x]} f \textrm{, where } \\
f&=\sum_{n=1}^{\infty} \sum_{\sigma \in \{0,1,2,3\}^n} n \frac{h(l_n)}{l_n}  \chi_{I_{\sigma}},
\end{align*}
where $\chi_{I_\sigma}$ denotes the characteristic function of $I_\sigma$.
Note that $f$ is integrable so $g$ is well-defined, because \eqref{eq:l_n} and \eqref{i:1} yield the estimate
$$\int_{[0,1]} f=\sum_{n=1}^{\infty} \sum_{\sigma \in \{0,1,2,3\}^n} n \frac{h(l_n)}{l_n} \lambda({I_{\sigma}})=
\sum_{n=1}^{\infty} 4^n n 5^{-n}< \infty.$$
Let
\[P = \bigcap_{n=1}^{\infty} \bigcup_{\sigma \in \{1,2\}^n}  I_{\sigma}.
\]
It is clear that $P$ is a non-empty perfect set. Now we show that $P$ satisfies the desired property.
Fix $p \in P$, $x\in [0,1]$ and $n\in \mathbb{N}^+$ such that $l_{n+1}<|x-p|\leq l_n$, it is enough to prove that
\begin{equation} \label{eq:main} \frac{g(x) -g(p)}{h(x-p)}\geq \frac{n}{25}. \end{equation}
Properties $|x-p|>l_{n+1}$ and \eqref{i:3} imply that there is a $\sigma\in \{0,1,2,3\}^{n+2}$ such that $I_{\sigma}$ is between $x$ and $p$.
The definition of $g$, \eqref{i:1}, \eqref{eq:l_n} and the monotonicity of $h$ yield
$$|g(x) -g(p)|\geq \int_{I_{\sigma}} f \geq (n+2) h(l_{n+2})=\frac{n+2}{25} h(l_{n})\geq  \frac{n+2}{25} h(|x-p|). $$
Clearly $g(x)-g(p)$ and $h(x-p)$ have the same sign, thus the above inequality implies \eqref{eq:main}, which concludes the proof.
\end{proof}

\begin{proof}[Proof of Theorem~\ref{t:infderiv}]
Applying Theorem~\ref{t:DW} for $A=B=\emptyset$ yields that $\mathcal{D}$ is not prevalent.
We next show that $\mathcal{D}$ is not shy. Let $\mathcal{K} \subset C[0,1]$ be a given compact set, by Lemma~\ref{l:nshy}
it is enough to prove that $\mathcal{K}+g\subset \mathcal{D}$ for some $g \in C[0,1]$.
Let $h \in C[0,1]$ be a strictly increasing function guaranteed by Lemma~\ref{l:iK}. We may assume that $h(x)\ge x$ for all $x \in [0,1]$.
Extend $h$ to $[-1,1]$ so that $h(-x) = -h(x)$ for all $x\in [0,1]$. Choose a non-empty perfect set $P\subset [0,1]$ and a non-decreasing function $g\in C[0,1]$ according to Lemma~\ref{l:ih}. The definitions of $h$ and $g$ imply that for all $f\in \mathcal{K}$ and $p\in P$ we have

\begin{align*}
\liminf_{x\to p} \frac{(f+g)(x) - (f+g)(p)}{x-p}&=\liminf_{x\to p} \left( \frac{f(x) - f(p)}{h(x-p)}  + \frac{g(x) - g(p)}{h(x-p)}\right) \frac{h(x-p)}{x-p}\\
&\geq \liminf_{x\to p} \left( -1+\frac{g(x) - g(p)}{h(x-p)}\right) \frac{h(x-p)}{x-p}\\
&\geq (-1+\infty)\cdot 1=\infty.
\end{align*}
Thus $(f+g)'(p) = \infty$ for all $f \in \mathcal{K}$ and $p \in P$, therefore $\mathcal{K}+g\subset \mathcal{D}$.
\end{proof}

\section{Open problems}

\begin{problem}
What can we say about the topological properties of the level sets of the prevalent/non-shy many $f \in C([0,1]^d)$ for $d \ge 2$?
\end{problem}

\begin{problem} Krasinkiewicz \cite{Kr} and Levin \cite{L} independently showed that if $K$ is a non-degenerate continuum then
the generic $f\in C(K)$ has the property that each component of each of its level sets is hereditarily indecomposable.
What can we say from the point of view of prevalence?
\end{problem}

\begin{problem} Buczolich and Darji \cite{BD} showed that if $K$ is a non-degenerate continuum then the
generic $f\in C(K)$ has the property that the Bruckner--Garg Theorem holds when $f^{-1}(y)$ is replaced by $\Comp (f^{-1}(y))$,
the space whose elements are the components of $f^{-1} (y)$ and the topology is the so called upper semicontinuous topology (that is, we consider the factor topology on $f^{-1} (y)$ where the equivalence classes are the components).
What can we say from the point of view of prevalence?
\end{problem}

\begin{problem}
Buczolich and Darji \cite{BD} examined the fiber structure of the generic map $f\in C(\mathbb{S}^2)$, where $\mathbb{S}^2$ is the two-dimensional sphere.
What can we say from the point of view of prevalence?
\end{problem}


\begin{problem}
What can we say if we replace $C(K)$ by $C(K, \mathbb{R}^d)=\{f \colon K \to \mathbb{R}^d  : f \textrm{ is continuous}\}$?
\end{problem}

For the generic version of this last question see e.g. \cite{Ka}.

\subsection*{Acknowledgments}
We are indebted to Y. Peres for some helpful conversations, especially for suggesting us the proof of Theorem~\ref{t:occ}. We also thank M. Vizer for some illuminating discussions.


\begin{thebibliography}{99}

\bibitem{BDE} R.~Balka, U.~B.~Darji, M.~Elekes, Hausdorff and packing dimension of fibers and graphs of prevalent continuous maps, \textit{Adv.\ Math.}\ \textbf{293} (2016), 221--274.

\bibitem{BG} A.~M.~Bruckner, K.~M.~Garg, The level set structure of a residual set of continuous functions, \textit{Trans.\ Amer.\ Math.\ Soc.}\
\textbf{232} (1977), 307--321.

\bibitem{BD} Z.~Buczolich, U.~Darji, Pseudoarcs, pseudocircles, Lakes of Wada and generic maps of $S^2$,
\textit{Topology Appl.}\ \textbf{150} (2005), 223--254.

\bibitem{C} J.~P.~R.~Christensen, On sets of Haar measure zero in abelian Polish groups,
\textit{Israel J.\ Math.}\ \textbf{13} (1972), 255--260.

\bibitem{DW} U.~B.~Darji, S.~C.~White, Haar ambivalent sets in the space of continuous functions,
\textit{Acta Math.\ Hungar.}\ \textbf{126} (2010), no.~3, 230--240.

\bibitem{D} R.~Dougherty, Examples of non-shy sets, \textit{Fund.\ Math.}\ \textbf{144} (1994), 73--88.

\bibitem{Fr} D.~H.~Fremlin, \textit{Broad functions}, 
Measure Theory, Volume 2, Torres Fremlin, 2003.

\bibitem{HSY} B.~Hunt, T.~Sauer, J.~Yorke, Prevalence: a translation-invariant ``almost
every'' on infinite-dimensional spaces, \textit{Bull.\ Amer.\ Math.\ Soc.}\ \textbf{27} (1992), 217--238.

\bibitem{Ka} H.~Kato, Higher-dimensional Bruckner-Garg type theorem, \textit{Topology Appl.}\ \textbf{154} (2007), no.~8, 1690--1702.

\bibitem{K} A.~S.~Kechris, \textit{Classical descriptive set theory}, Springer-Verlag, 1995.

\bibitem{Kr} J.~Krasinkiewicz, On mappings with hereditarily indecomposable fibers, \textit{Bull.\ Polish
Acad.\ Sci.\ Math.}\ \textbf{44} (1996), no.~2, 147--156.

\bibitem{L} M.~Levin, Bing maps and finite dimensional maps, \textit{Fund.\ Math.}\ \textbf{151} (1996), no.~1, 47--52.

\bibitem{MP} P.~M\"orters, Y.~Peres, \textit{Brownian Motion}, with an appendix by
Oded Schramm and Wendelin Werner, Cambridge University Press, 2010.

\bibitem{S} K.~Simon, Some dual statements concerning Wiener measure and Baire category, \textit{Proc.\ Amer.\ Math.\ Soc.}\ \textbf{106} (1989), no.~2, 455--463.

\bibitem{Z} L.~Zaj\'{\i}\v{c}ek, On differentiability properties of typical continuous functions and Haar null sets,
\textit{Proc.\ Amer.\ Math.\ Soc.}\ \textbf{134} (2005), no.~4, 1143--1151.

\end{thebibliography}
\end{document}